\documentclass[openany, amssymb, psamsfonts]{amsart}
\usepackage{mathrsfs,comment}
\usepackage[usenames,dvipsnames]{color}
\usepackage[normalem]{ulem}
\usepackage{url}
\usepackage[all,arc,2cell]{xy}
\usepackage{tikz-cd}
\UseAllTwocells
\usepackage{enumerate}
\usepackage{mathabx}
\usepackage{hyperref}  
\usepackage{proof-at-the-end}
\hypersetup{%
  bookmarksnumbered=true,%
  bookmarks=true,%
  colorlinks=true,%
  linkcolor=blue,%
  citecolor=blue,%
  filecolor=blue,%
  menucolor=blue,%
  pagecolor=blue,%
  urlcolor=blue,%
  pdfnewwindow=true,%
  pdfstartview=FitBH}

\def\makeautorefname#1#2{\expandafter\def\csname#1autorefname\endcsname{#2}}
\def\equationautorefname~#1\null{(#1)\null}
\makeautorefname{footnote}{footnote}%
\makeautorefname{item}{item}%
\makeautorefname{figure}{Figure}%
\makeautorefname{table}{Table}%
\makeautorefname{part}{Part}%
\makeautorefname{appendix}{Appendix}%
\makeautorefname{chapter}{Chapter}%
\makeautorefname{section}{Section}%
\makeautorefname{subsection}{Section}%
\makeautorefname{subsubsection}{Section}%
\makeautorefname{theorem}{Theorem}%
\makeautorefname{thm}{Theorem}%
\makeautorefname{cor}{Corollary}%
\makeautorefname{lem}{Lemma}%
\makeautorefname{prop}{Proposition}%
\makeautorefname{pro}{Property}
\makeautorefname{conj}{Conjecture}%
\makeautorefname{defn}{Definition}%
\makeautorefname{notn}{Notation}
\makeautorefname{notns}{Notations}
\makeautorefname{rem}{Remark}%
\makeautorefname{quest}{Question}%
\makeautorefname{exmp}{Example}%
\makeautorefname{ax}{Axiom}%
\makeautorefname{claim}{Claim}%
\makeautorefname{ass}{Assumption}%
\makeautorefname{asss}{Assumptions}%
\makeautorefname{con}{Construction}%
\makeautorefname{prob}{Problem}%
\makeautorefname{warn}{Warning}%
\makeautorefname{obs}{Observation}%
\makeautorefname{conv}{Convention}%

\newtheorem{thm}{Theorem}[section]
\newtheorem{cor}{Corollary}[section]
\newtheorem{prop}{Proposition}[section]
\newtheorem{lem}{Lemma}[section]

\theoremstyle{definition}
\newtheorem{defn}{Definition}[section]

\newtheorem{rem}{Remark}[section]

\makeatletter
\let\c@obs=\c@thm
\let\c@cor=\c@thm
\let\c@prop=\c@thm
\let\c@lem=\c@thm
\let\c@prob=\c@thm
\let\c@con=\c@thm
\let\c@conj=\c@thm
\let\c@defn=\c@thm
\let\c@notn=\c@thm
\let\c@notns=\c@thm
\let\c@exmp=\c@thm
\let\c@ax=\c@thm
\let\c@pro=\c@thm
\let\c@ass=\c@thm
\let\c@warn=\c@thm
\let\c@rem=\c@thm
\let\c@sch=\c@thm

\makeatother

\newcommand{\C}{\mathbb C}

\renewcommand{\H}{\mathbb H}
\newcommand{\Q}{\mathbb Q}
\newcommand{\Z}{\mathbb Z}

\newcommand{\R}{\mathbb R}

\newcommand{\fs}{\mathfrak s}
\newcommand{\fp}{\mathfrak p}
\newcommand{\fm}{\mathfrak m}

\newcommand{\cA}{\mathcal A}
\newcommand{\cB}{\mathcal B}

\newcommand{\cL}{\mathcal L}

\newcommand{\cG}{\mathcal G}

\newcommand{\RNum}[1]{\uppercase\expandafter{\romannumeral #1\relax}}

\newcommand{\Gal}{\mathrm {Gal\ }}

\newcommand{\grad}{\mathrm{grad}}

\title{Counting hyperbolic 4-manifolds with vanishing Seiberg-Witten invariants}

\author{Kaixu Zhang}

\begin{document}

\begin{abstract}
Ian Agol and Francesco Lin proved the existence of hyperbolic four-manifolds with vanishing Seiberg-Witten invariants. We prove that the number of such manifolds of volume at most $v$ is asymptotically bounded by $v^{cv}$ considered up to commensurability, which has the same form as the lower bound and upper bound of the number of hyperbolic four-manifolds of volume at most $v$ proved by Tsachik Gelander and Arie Levit.
\end{abstract}

\maketitle

\tableofcontents

\section{Introduction}  
Seiberg-Witten theory gives rise to a powerful interplay between the geometry and topology of smooth 4-manifolds. Witten \cite{witten1994monopolesfourmanifolds} proved that if a 4-manifold with $b_2^+\geq 2$ admits a metric of positive scalar curvature, then all its Seiberg-Witten invariants vanish. In \cite[Conjecture 1.1]{lebrun2001hyperbolicmanifoldsharmonicforms}, Claude LeBrun asked whether the Seiberg-Witten invariants of compact hyperbolic 4-manifolds vanish, based on the estimates of the scalar curvature and Weyl curvature for the spin$^c$ structure $\fs$ with non-trivial Seiberg-Witten invariants. 

\begin{prop}\cite[Theorem 3.5]{lebrun2001hyperbolicmanifoldsharmonicforms}

Let $M^4$ be a smooth compact oriented $4$-manifold with $b_2^+\geq 2$, and suppose that $\fs$ is a spin$^c$ structure with non-trivial Seiberg-Witten invariant. Let $g$ be any Riemannian metric on $M$, and let $\phi$ be a $g$-self-dual harmonic 2-form with de Rham class $[\phi]\,\in H^2(M,\R)$. Let $s_g$ be the scalar curvature and $w_g$ be the lowest eigenvalue of the self-dual Weyl curvature $W^+$ of $g$. Then the function $f=\sqrt{|\phi|}$ satisfies\[
\int_M\left(\frac{2}{3}s_g+2w_g\right)|\phi_g|d\mu_g+4\int_M|df|^2_gd\mu_g\leq (4\pi\sqrt{2})c_1(\fs)\cdot\mathrm{[}\phi\mathrm{]}
\]
\end{prop}

So the Seiberg–Witten basic classes satisfy stronger constraints for hyperbolic manifolds $(M^4,g)$ with the scalar curvature $s_g=-12$ and the Weyl curvature $W_g=0$.
 
In \cite[Theorem 1.1]{agol2018hyperbolicfourmanifoldsvanishingseibergwitten}, Ian Agol and Francesco Lin proved the existence of hyperbolic four-manifolds with vanishing Seiberg-Witten invariants, and remarked that there are infinitely many commensurability classes of arithmetic hyperbolic 4-manifolds containing representatives with vanishing Seiberg-Witten invariants.

 Let $C(v)$ denote the number of commensurability classes of hyperbolic manifolds admitting a
representative of volume at most $v$. In \cite[Theorem 1.2]{gelander2014countingcommensurabilityclasseshyperbolic}, Tsachik Gelander and Arie Levit proved that there are positive constants $a$ and $b$, such that $v^{av}\leq C(v)\leq v^{bv}$ for all $v$ sufficiently large. Inspired by their methods, we show that
the number of commensurability classes of hyperbolic manifolds admitting a representative with vanishing Seiberg-Witten invariants of volume at most $v$ has the same form:
\begin{thm}\label{1.1}
Let $VC(v)$ denote the number of commensurability classes of hyperbolic manifolds admitting a
representative with vanishing Seiberg-Witten invariants of volume at most $v$. Then there exist positive constants $b$ and $c$ such that \[
v^{cv}\leq VC(v)\leq v^{bv}
\]
for all $v$ sufficiently large.
\end{thm}

Theorem \hyperref[1.1]{1.2} is proved by embedding the $L$-space $M_5$ as a totally geodesic non-separating hypersurface into non-commensurable hyperbolic 4-manifolds, using the main result of Kolpakov–Reid–Slavich\cite{Kolpakov_2018}. By modifying Gelander-levit's approach in \cite{gelander2014countingcommensurabilityclasseshyperbolic}, most of our work is constructing appropriate building blocks, each of whose boundaries is a union of totally geodesic copies of the $L$-space $M_5$. By gluing them according to decorated graphs, we obtain infinitely many non-commensurable compact hyperbolic 4-manifolds with the non-separating $L$-space $M_5$, which satisfy the condition of \cite[Corollary 2.5]{agol2018hyperbolicfourmanifoldsvanishingseibergwitten}, .

~\\

In Section 2, we recall the basic construction of the Seiberg-Witten equations and the definitions of Seiberg-Witten invariants for closed 4-manifolds and monopole floer homology groups for closed 3-manifolds. Then we present a vanishing criterion for Seiberg-Witten invariants involving the $L$-space as a separating hypersurface.

In Section 3, we present the rough outline of embedding arithmetic hyperbolic 3-manifolds as a totally geodesic hypersurface into hyperbolic 4-manifolds. Here certain technical details are required in order to ensure compactness and boundary control.

In Section 4, we recall the construction in \cite{gelander2014countingcommensurabilityclasseshyperbolic} by establishing connections between interbreeding of hyperbolic manifolds and decorated graphs. We also explain how to construct hyperbolic manifolds from decorated graphs and show that commensurable manifolds have isomorphic decorated graphs.

In Section 5, we recall the examples by \cite{agol2018hyperbolicfourmanifoldsvanishingseibergwitten}. Then we construct suitable building blocks and check that the infinitely many manifolds obtained by these building blocks satisfy the condition of Theorem \hyperref[1.1]{1.2}.

\section{Seiberg-Witten invariants and a vanishing criterion}

We briefly review the necessary background on the Seiberg-Witten theory. Most of our discussions and notations are based on \cite{cbaf2759-40f3-37dd-8d44-27af0e7bad3a} and \cite{Kronheimer_Mrowka_2007}.

Let $X$ be an oriented 4-dimensional Riemannian manifold. a spin$^c$ structure is a hermitian vector bundle $S_X\rightarrow X$ of rank 4, with a Clifford multiplication\[
\rho: TX\rightarrow \mathrm{Hom}(S_X,S_X),\]
such that at each $x\in X$ we can find an oriented orthonormal frame $e_0,...e_3$ with
\[
\rho(e_0)=\begin{bmatrix}
0& -I_2\\
I_2 & 0 \\
\end{bmatrix}, \qquad \rho(e_i)=\begin{bmatrix}
0 & -\sigma_i^*\\
\sigma_i & 0\\ 
\end{bmatrix} \quad (i=1,2,3)
\]
in some orthonormal basis of the fiber $S_x$. Here $I_2$ is the $2\times2$ identity matrix and $\sigma_i$ is the Pauli matrices. If we extend Clifford multiplication to complex forms, then in the same basis for $S_x$ we have
\[
\rho(\mathrm{vol}_x)=\begin{bmatrix}
-I_2& 0\\
0 & I_2\\
\end{bmatrix}\]where $\mathrm{vol}=e_0\wedge e_1\wedge e_2\wedge e_3$ is the oriented volume form. So the eigenspaces of  $\rho(\mathrm{vol})$ give a decomposition of $S_X$ into two orthogonal rank-2 bundles.

Let $X$ be an oriented compact Riemannian $4$-manifold with a spin$^c$ structure $\fs_X=(S_X,\rho)$ and corresponding $-1$-eigenspace $S_X^+$ of the Clifford multiplication $\rho(\mathrm{vol}_x)$. The Seiberg-Witten equations associated to the spin$^c$ structure $\fs_X$ are equations for a pair $(A,\Phi)$ consisting of a spin$^c$ connection $A$ and a section $\Phi$ of the associated spin bundle $S_X^+$. The equations are the following:
\begin{equation}
\left\{
\begin{array}{l}
D^+_A\Phi=0  \\
F^+_{A^t}=(\Phi\Phi^*)_0+\eta
\end{array}.
\right.
\end{equation}
Here $F^+_{A^t}$ is the self-dual part of the curvature 2-form $F_{A^t}$ of the connection $A^t$, $(\Phi\Phi^*)_0$ denotes the trace-free part of the hermitian endomorphism $\Phi\Phi^*$ of the bundle $S^+_X$ and $\eta$ is some perturbation of $2$-form. A solution $(A,\Phi)$ is called reducible if $\Phi=0$, and irreducible otherwise.

The gauge group $\cG$ of $X$ is the group of unitary bundle automorphisms of $S_X$ which commutes with Clifford multiplication, and it can be identified with the group of $S^1$-valued functions $u:X\rightarrow S^1$, acting by scalar multiplication. The set of solutions $(A,\Phi)$ is invariant under the action of the gauge group, and we write $N(X, \fs_x)$ for the quotient space of the set of solutions of the equations (1) by the action of $\cG$. Some properties of the moduli space $N(X, \fs_X)$ are listed in the following theorem.

\begin{thm}\cite[Theorem 1.4.4]{Kronheimer_Mrowka_2007},  
Suppose that the perturbation $\eta$ is chosen so that the moduli space $N(X,\fs_X)$
 is regular, i.e. the linearization of the equations is a surjective linear operator for all solutions $(A,\Phi)$, and that $N(X,\fs_X)$ contains no reducible solutions, as we can always do when $b_2^+(X)\geq1$. Then the moduli space $N(X,\fs_X)$ is a smooth compact manifold, whose dimension $d$ is given by the formula\[
 d=\frac{1}{4}\left(c_1(S^+_X)^2[X]-2\lambda(X)-3\sigma(X)\right)\]

\end{thm}

We denote by $\cB(X,\fs_X)$ the quotient space of $\cA\times\Gamma(S_X^+)$ by the action of $\cG$, and $B^*(X,\fs_X)$ is the irreducible part. After choosing a perturbation $\eta$, we assume $N(X,\fs_X)\subset\cB^*(X,\fs_X)$, and if we further give a homology orientation, we have a well-defined class $[N(X,\fs_X)]\in H_d(\cB^*(X,\fs_X))$.

Since $\cB^*(X,\fs_X)$ is the quotient space of $\cA\times(\Gamma(S^+_X)\setminus\{0\})$ by the free action of $\cG$, the quotient map\[
\cA\times(\Gamma(S^+_X)\setminus\{0\})\rightarrow\cB^*(X,\fs_X)
\]
is a principal $\cG$-bundle. And if we choose a basepoint $x_0\in X$, then we obtain a homomorphism $\cG\rightarrow S^1$ by evaluation at $x_0$, and there is an associated principal $S^1$ bundle $P\rightarrow \cB^*(X,\fs_X)$. Therefore, there is a well-defined 2-cohomology class\[
u=c_1(P)\in H^2(\cB^*(X,\fs_X);\Z).\]
And the Seiberg-Witten invariants $\fm(X,\fs_X)\in \Z$ are defined by the formula\[
\fm(X,\fs_X)=\left\{
\begin{array}{lr}
\langle u^{\frac{d}{2}},[N(X,\fs_X)]\rangle, & \mathrm{when}\;d \;\mathrm{is\; even},\\
0, & \mathrm{when} \;d\; \mathrm{is \;odd}.
\end{array}
\right.
\]  

Let $Y$ be a closed, connected, oriented Riemannian 3-manifold. For each isomorphism class of spin$^c$ structure $\fs$ on $Y$, we choose a reference spin$^c$ connection $B_0=B_0(\fs)$ on a spin bundle $S=S(\fs)$. The Chern-Simons-Dirac function of a spin$^c$ connection $B$ and a section $\Psi$ of the corresponding spin bundle, defined by\[
\cL(B,\Psi)=-\frac{1}{8}\int_Y(B^t-B_0^t)\wedge(F_{B^t}+F_{B^t_0})+\frac{1}{2}\int_Y\langle D_B\Psi,\Psi\rangle d\mathrm{vol},
\]
and we obtain the gradient of the function\[
\grad\, \cL=\left((\frac{1}{2}*F_{B^t}+\rho^{-1}(\Psi\Psi^*)_0)\otimes 1_S, D_B\Psi\right).
\]

We can also define the gauge group $\cG$ of $Y$ as the group of unitary bundle automorphisms of $S_X$ which commutes with Clifford multiplication. The downward gradient gives a flow on the blow-up $\cB^\sigma(Y,\fs)$ of the quotient space $\cB(Y,\fs)=(\cA\times\Gamma(S))/\cG$, and after perturbing the function to achieve the transversality, we obtain three complexes $\hat{C}$, $\bar{C}$ and $\check{C}$, corresponding to some combinations of boundary-stable, boundary-unstable and interior critical points up to grading shifts. So it is defined for each spin$^c$ structure on $Y$ the monopole Floer homology groups corresponding the three complexes, and they fit into the long exact sequence of graded $\Z[U]$-modules\[
\begin{tikzcd}
\cdots \arrow[r] & {\overline{HM}_*(Y,\fs)} \arrow[r, "i_*"] & {\widecheck{HM}_*(Y,\fs)} \arrow[r, "j_*"] & {\widehat{HM}_*(Y,\fs)} \arrow[r, "p_*"] & {\overline{HM}_*(Y,\fs)} \arrow[r] & \cdots
\end{tikzcd}
\]
where $U$ has degree $-2$. The reduced group $HM_*(Y,\fs)$ is defined as the image of $j_*$ in $\widehat{HM_*}(Y,\fs)$. And when $Y$ is a rational homology sphere, we have an identification of $\Z[U]$-modules (up to grading shift) with the Laurent series(\cite[Proposition 35.3.1]{Kronheimer_Mrowka_2007})\[
\overline{HM}_*(Y,\fs)\cong\Z[U^{-1},U].\]

\begin{defn}\cite[Definition 9.1]{kronheimer2004monopoleslensspacesurgeries}, \label{2.3}
We say that a rational homology sphere $Y$ is an $L$-space if $j_*$ is trivial for all spin$^c$ structures.
\end{defn}

So for an $L$-space $Y$, $HM_*(Y, \fs)=0$ for all spin$^c$ structures $\fs$.

\begin{prop}\cite[Proposition 3.11.1]{Kronheimer_Mrowka_2007}
Let $X$ be a closed, oriented 4-manifold with $b_2^+(X)\geq 2$, and suppose that $X=X_1\cup X_2$ with $\partial X_1=-\partial X_2=Y$, a connected 3-manifold. If $Y$ is an $L$-space, and $b_2^+(X_1)$ and $b_2^+(X_2)$ are both positive, then $m(X, \fs_X)=0$ for all spin$^c$ structures $\fs_X$. 
\end{prop}

\begin{proof}
 Since $b_1(Y)=0$, a spin$^c$ structure $\fs_X$ on $X$ is determined by the restrictions $\fs_i=\fs_X|_{X_i}$. This follows from the injectivity of the map $H^2(X;\Z)\rightarrow H^2(X_1;\Z)\oplus H^2(X_2;\Z)$ in the Mayer-Vieoris sequence and that these groups classify spin$^c$ structures. Let $\fs=\fs_X|_Y$. Then it is sufficient to show that $\fm(u|X,\fs_X)=0$ for classes $u=u_1u_2$ where $u_i$ is a cohomology class in the configuration space $\cB(X_i,\fs_i)$ of $X_i$. A cobordism $W$ from $Y_0$ to $Y_1$ induces a map in long exact sequences of monopole floer homology groups (\cite[Section 3.4]{Kronheimer_Mrowka_2007}) and if $b_2^+(W)\geq 1$, we have that $\overline{HM}_*(u|W,\fs)=0$ (\cite[Proposition 3.5.2]{Kronheimer_Mrowka_2007}). So we can define the relative invariant $\psi_{(u_1| X_1,\fs_1)}\in \widehat{HM}_*(Y, \fs)$: let $W_1$ be the cobordism obtained from $X_1$ by removing a ball, and consider the induced map\[
 \widehat{HM}_*(u_1|X_1,\fs_1):\widehat{HM}_*(S^3)=\Z[U]\rightarrow\widehat{HM}_*(Y,\fs).\]
 Then $\psi_{(u_1|X_1,\fs_1)}=\widehat{HM}_*(u_1|W_1,\fs_1)(1)$. We consider the commutative diagram\[
 \begin{tikzcd}
\widehat{HM}_*(S^3) \arrow[r, "p_*"] \arrow[d, "{\widehat{HM}_*(u_1|W_1,\fs_1)}"'] & \overline{HM}_*(S^3) \arrow[d, "{\overline{HM}_*(u_1|W_1,\fs_1)}"] \\
{\widehat{HM}_*(Y,\fs)} \arrow[r, "p_*"]                                           & {\overline{HM}_*(Y,\fs)}                                          
\end{tikzcd}
 \] 
and as $b_2^+(W_1)\geq 1$, the vertical map on the right vanishes, which implies that $\psi_{(u_1|X_1,\fs_1)}\in \ker(p_*)=HM_*(Y,\fs)$. Similarly, using the map induced in homology by $W_2$, we obtain an element $\psi_{(u_2|X_2,\fs_2)}\in HM_*(-Y,\fs)$, which is identified with $HM^*(Y,\fs)$. By the general gluing theorem in \cite[Equation 3.22]{Kronheimer_Mrowka_2007},\[
\fm(u|X,\fs_X)=\langle\psi_{(u_1|X_1,\fs_1)},\psi_{(u_2|X_2,\fs_2)}\rangle,\]
where the angular brackets denote the natural pairing\[
HM_*(Y,\fs)\times HM^*(Y,\fs)\rightarrow \Z.\]
In our assumpions, the group $HM_*(Y,\fs)$ vanishes, so this pairing is zero, and $\fm(u|X,\fs_X)$ is zero for all spin$^c$ structures $\fs_X$.

\end{proof}

The construction of known examples by Agol and Lin  \cite{agol2018hyperbolicfourmanifoldsvanishingseibergwitten} fundamentally relies on the following corollary.

\begin{cor}\cite[Corollary 2.5]{agol2018hyperbolicfourmanifoldsvanishingseibergwitten}\label{2.4}
Suppose $X$ is a 4-manifold with $b_2^+(X)\geq 1$ which admits an embedded non-separating $L$-space $Y$. Then $X$ admits infinitely many covers which have all vanishing Seiberg-Witten invariants.
\end{cor}

\begin{proof}
We consider the double cover $\tilde{X}$ of $X$ formed by gluing together two copies $W_1$ and $W_2$ of the cobordism from $Y$ to obtained by cutting $X$ along $Y$. Consider a properly embedded path $\gamma\subset W_1$ between the two copies of $Y$, and denote by $T$ its tubular neighborhood. Then we have the decomposition $X=$($W_1\setminus T$)$\cup$($W_2\cup T$), where the two manifolds are glued along a copy of $Y\#\overline{Y}$. 

By \cite[Section 4]{Lin_2017}, the Heegaard floer chain complexes have vanishing differentials in a suitable sense, so the connected sum also gives a chain complex with trivial differential, which implies the Heegaard Floer homology is again minimal and $Y\#\overline{Y}$ is an $L$-space. And both $W_1\setminus T$ and $W_2\cup T$ have $b_2^+\geq 1$, so we conclude. 
\end{proof}

\section{Embedding arithmetic hyperbolic manifolds}
In this section, we discuss the embedding arithmetic hyperbolic 3-manifolds as totally geodesic hypersurfaces into arithmetic hyperbolic 4-manifolds in \cite{Kolpakov_2018}. We first review the definitions of admissible quadratic forms and arithmetic groups of simplest type.

Let $k$ be a totally real number field of degree $d$ over $\Q$ with a fixed embedding into $\R$ and the ring of integers $R_k$, and let $V$ be an $(n+1)$-dimensional vector space over $k$ equipped with a non-degenerate quadratic form f defined over $k$ which has signature $(n,1)$ at the fixed embedding, and signature $(n+1,0)$ at the remaining $d-1$ embedding. 

The quadratic form f is equivalent over $\R$ to the quadratic form $x_0^2+x_1^2+\cdots+x_{n-1}^2-x_n^2$, and for any non-identity Galois embedding $\sigma:k\rightarrow\R$, the quadratic form f$^\sigma$ is equivalent over $\R$ to $x_0^2+x_1^2+\cdots+x_{n-1}^2+x_n^2$. We call such a quadratic form \textit{admissible}.

Let $F$ be the symmetric matrix associated with the quadratic form f and let O(f) denote the linear algebraic groups defined over $k$ defined as:\[
\begin{split}
  &\mathrm{O(f)}=\{X\in \mathrm{GL}(n+1,\C):X^tFX=F\}\quad \mathrm{and} \\
 &\mathrm{SO(f)}=\{X\in \mathrm{SL}(n+1,\C):X^tFX=F\}  
\end{split}
\] 

Let $J_n$ denote both the quadratic form $x_0^2+x_1^2+\cdots x_{n-1}^2-x_n^2$, and the diagonal matrix associated with the form. The hyperbolic space $\H^n$ can be identified with the upper half-sheet $\{x\in \R^{n+1}:J_n=-1,\,x_n>0\}$ of the hyperboloid, and we let\[
\mathrm{O(n,1)}=\{X\in \mathrm{GL}(n+1,\R):X^t J_n X=J_n\}.\]
We can also identify Isom($\H^n$) with the subgroup of O(n,1) preserving the upper half-sheet of the hyperboloid $\{x\in\R^{n+1}:J_n=-1\}$, denoted by O$^+$(n,1).

\begin{defn}
Let $G$ be a group, $H_1$, $H_2\leq G$ be subgroups. We say that $H_1$ is commensurable in $G$ with $H_2$ if $[H_1:H_1\cap H_2]<\infty$, $[H_2:H_1\cap H_2]<\infty$.
\end{defn}

Given an admissible quadratic form defined over $k$ of signature $(n,1)$, there exists $T\in\mathrm{GL(n+1},\R)$ such that $T^{-1}\mathrm{O(f,}\R)T=\mathrm{O(n,1)}$.
\begin{defn}
A subgroup  $\Gamma<$ O$^+$(n,1) is called arithmetic of simplest type if $\Gamma$ is commensurable with the image in O$^+$(n,1) of an arithmetic subgroup of O(f) under the conjugation map above.
\end{defn}
\begin{defn}
Let $G$ be a group. Then $G^{(2)}=\langle g^2|g\in G\rangle$.
    
\end{defn}

The following proposition allows one to embed an arithmetic group in arithmetic groups with higher dimensions.

\begin{prop}\cite[Corollary 4.2]{Kolpakov_2018}\label{3.4}
Let $\Gamma$ be an arithmetic subgroup of $\mathrm{O}^+\mathrm{(n,1)}$ of simplest type arising from an admissible quadratic form $\mathrm{f}$ of signature $(n,1)$ defined over a totally real field $k$. Suppose that there is an admissible quadratic form $\mathrm{g}$ of signature $(n+1,1)$ defined over the same field $k$, with $\mathrm{O(f)<O(g)}$. Then:
\begin{enumerate}[(1)]
\item If $n$ is even, $\Gamma$ embeds in an arithmetic subgroup of $\mathrm{O^+(n+1,1)}$ of simplest type.
\item If $n$ is odd, $\Gamma^{(2)}$ embeds in an arithmetic subgroup of $\mathrm{O^+(n+1,1)}$ of simplest type.
\end{enumerate}
\end{prop}

 From the above proposition, we can simplify the problem to find an admissible quadratic form g so that $\mathrm{O(f)<O(g)}$. Here we need only to consider the case $k\neq \Q$ for our purpose.

\begin{lem}\cite[Proposition 5.1]{Kolpakov_2018}
Suppose that $\mathrm{f}$ is represented by the admissible diagonal quadratic form $a_0x_0^2+a_1x_1^2+\cdots+a_{n-1}x_{n-1}^2-bx_n^2$ over the field $k\neq \Q$, where $a_i\in R_k$ are all positive and square free for $i=0,...,n-1$, and $b\in R_k$ is positive and square free. Then there is an admissible diagonal quadratic form $\mathrm{g}$ of signature $(n+1,1)$ with $\mathrm{O(f)<O(g)}$.
\end{lem}
\begin{proof}
If f is anisotropic over $k$, that is, f$(x)=0$ if and only if $x=0$, then we can assume that $b\neq a_i$ for $i=0,...,n-1$. Since O($\lambda$f)=O(f) for all $\lambda\in k^*$, we can multiply f by $a_0^{-1}$ and assume that $a_0=1$, and also that all coefficients are square-free. Then we can take g$=dy^2+$f so that $d$ is square-free in $k$, which will be a quadratic form over $k$. And O(g,$R_k$) is cocompact, as follows from \cite[Proposition 6.4.4]{morris2015introductionarithmeticgroups}.
 \end{proof}

 \begin{rem}
 The above lemma is also true for $k=\Q$, but we don't need to discuss this case here.
 \end{rem}

The general case can be reduced to the diagonal case by taking some $T\in \mathrm{GL(n+1},k)$ so that $T^{-1}\mathrm{O(f)}T=\mathrm{O(f_0)}$ for some admissible diagonal quadratic form $\mathrm{f_0}$ defined over $k$(see \cite{lam2005algebraic}). Then we can extend $T$ to define a matrix\[
\widehat{T}=\left(
\begin{array}{l}
1\quad|\quad0\\
0\quad|\quad T
\end{array}
\right)\in \mathrm{GL(n+2},k)
\]  
which provides an equivalence of the diagonal form $\mathrm{g_0}$ to an admissible quadratic form g with $\mathrm{O(f)<O(g)}$.

Let $M=\H^n/\Gamma$ be an orientable arithmetic hyperbolic $n$-manifold of simplest type and if $n$ is odd, we replace $M$ by $M^{(2)}=\H^n/\Gamma^{(2)}$. By Proposition \hyperref[3.4]{3.4}, there exists an arithmetic lattice $\Lambda$ of simplest type in $\mathrm{SO^+(n+1,1)}$ such that $\Gamma < \Lambda$. Then we can find a torsion-free subgroup $\Lambda_1<\Lambda$ with $\Gamma<\Lambda_1$ and $\Lambda_1$ is GFERF  (\cite[Proposition 7.1 and Theorem 6.2]{Kolpakov_2018}) . In this case, Scott\cite{Scott1978} proved that $M$ embeds in a finite sheeted cover of $\H^{n+1}/\Lambda_1$. Combining these results, we obtain the following theorem:
\begin{thm}\label{3.7}\cite[Theorem 1.1]{Kolpakov_2018}
Let $M=\H^n/\Gamma$ \rm{(}$n\geq 2$\rm{)}  be an orientable arithmetic hyperbolic $n$-manifold of simplest type.
\begin{enumerate}
\item If $n$ is even, $M$ embeds as a totally geodesic submanifold of an orientable arithmetic hyperbolic $(n+1)$-manifold $W$.
\item  If $n$ is odd, the manifold $M^{(2)}=\H^n/\Gamma^{(2)}$ embeds as a totally geodesic submanifold of an orientable arithmetic hyperbolic $(n+1)$-manifold $W$.
\end{enumerate}

Moreover, when $M$ is not defined over $\Q$ (and is therefore closed), the manifold $W$ can be taken to be closed.
\end{thm}

Here the final sentence follows from the fact that closed arithmetic manifolds of simplest type are associated with quadratic forms either over a finite extension $k$ of $\Q$, $k\neq \Q$, or with quadratic forms over $\Q$
which are anisotropic.

\section{Decorated Graphs and Building Blocks}

In this section, we present the construction methods in \cite{gelander2014countingcommensurabilityclasseshyperbolic}, involving decorated graphs and building blocks.

Let $F$ denote the free group generated by two non-commuting elements $\{a,b\}$. The \textit{Schreir graph} $\Gamma_H$ corresponding to a subgroup $H\leq F$ is the quotient of the Cayley graph of $F$ by the natural action of $H$. Thus a Schreir graph is a 4-regular graph with oriented edges that are labeled by the set $\{a^{\pm},b^{\pm}\}$.

\begin{defn}\cite[Definition 2.1]{gelander2014countingcommensurabilityclasseshyperbolic}
A decorated graph is a 4-regular graph $\Gamma$ with oriented edges labeled by $\{a^{\pm},b^{\pm}\}$ whose vertices are 2-colored, and we will refer to each vertex as either colored or not.

A covering map of decorated graphs is a topological graph covering that preserves both the edge orientations and labels, and the vertex coloring.
\end{defn}

Decorated graphs with a single colored vertex are exactly Schreir graphs for finite index subgroups of $F_2$, since we can identify the single colored vertex with the identity element $\bar{e}$ in the quotient group, and obtain the corresponding finite index subgroup $H$ by the covering space theory.

\begin{defn}\cite[Definition 3.1]{gelander2014countingcommensurabilityclasseshyperbolic}
 The building blocks are the six given manifolds with boundaries $V_0$,$V_1$,$A^+$,$A^-$,$B^+$,$B^-$, which satisfy the following properties:
 \begin{itemize}
 \item Each is a complete real hyperbolic $n$-dimensional manifold of finite volume with totally geodesic boundary.

 \item $V_0$ and $V_1$ have 4 boundary components each, while $A^{\pm}$ and $B^{\pm}$ have 2 boundary components each.

 \item  Every boundary component of any of the above manifolds is isometric to a fixed $(n-1)$-dimensional complete finite-volume manifold $N$.

 \item  The six manifolds are embedded in respectively six manifolds without boundary, which are arithmetic and pairwise non-commensurable.
 \end{itemize}

 And given a decorated graph $\Delta$, we let $M_{\Delta}$ denote a manifold obtained by associating a copy of either $V_0$ or $V_1$ for each vertex in $\Delta$ according to its color, and a copy of the pair $A^+$ and $A^-$ or the pair $B^+$ and $B^-$ for every edge of $\Delta$ according to its label and orientation, and gluing them according to the graph incidence relation by identifying corresponding isometric copies of $N$.

 We refer to the isometric copies of $V_0,...,B^-$ inside $M_{\Delta}$ as the building block submanifolds.

\end{defn}

Working with decorated graphs will be useful in ruling out common covering spaces.

\begin{prop}\cite[Proposition 2.2]{gelander2014countingcommensurabilityclasseshyperbolic}\label{4.3}

Let $\Gamma_1$ and $\Gamma_2$ be two finite decorated graphs, each having a single colored vertex. If $\Gamma_1$ and $\Gamma_2$ are not isomorphic then they do not have a common decorated cover.
\end{prop}

\begin{proof}

Since $\Gamma_1$ and $\Gamma_2$ each have a single colored vertex, we may regard them as $\Gamma_{H_i}$ for some finite-index subgroup $H_i\leq F$, $i=1,2.$ Since the two graphs are not isomorphic, we have that $H_1\neq H_2$ as subgroups of $F$.

Assume by contraction $\Gamma_1$ and $\Gamma_2$ have a common decorated cover $\bar{\Gamma}$ with covers $p_i:\bar{\Gamma}\rightarrow \Gamma_i$. We consider some loop $\gamma$ in $\Gamma_1$ based at the colored vertex such that $l(\gamma)\in H_1\setminus H_2$, and let $\bar{\gamma}$ be a lift of $\gamma$ to $\bar{\Gamma}$ and let $x\in \bar{\Gamma}$ denote the end-point of $\bar{\gamma}$. Since $\gamma$ is a loop, the end-point of $p_1\circ\bar{\gamma}=\gamma$ is colored, while the end-point of $p_2\circ \bar{\gamma}$ is not colored since $l(\gamma)\notin H_2$. This contradicts the fact that $p_1$ and $p_2$ are assumed to preserve the decorated structure since both end-points are covered by $x\in\bar{\Gamma}$.
\end{proof}

The following proposition generalizes Proposition \hyperref[4.3]{4.3}, which reduces the problem of constructing non-commensurable manifolds to the problem of constructing non-isomorphic finite decorated graphs.

\begin{prop}\cite[Proposition 3.3]{gelander2014countingcommensurabilityclasseshyperbolic}\label{4.4}
Let $\Delta_1$ and $\Delta_2$ be two finite decorated graphs, each having a single colored vertex. If $\Delta_1$ and $\Delta_2$ are not isomorphic then the manifolds $M_{\Delta_1}$ and $M_{\Delta_2}$ are not commensurable.
\end{prop}

\begin{proof}[Sketch of Proof]
Suppose, by way of contradiction, that $M$ is a common finite cover of both $M_{\Delta_1}$ and $M_{\Delta_2}$ with associated covering maps $\pi_i:M\rightarrow M_{\Delta_i}$. Let $x\in M$ be a point. We can prove that $\pi_1(x)$ belongs to the interior of some building block sub-manifold of $M_{\Delta_1}$ if and only if $\pi_2(x)$ belongs to the interior of a building block of the same type in $M_{\Delta_2}$(see \cite[Lemma 3.5]{gelander2014countingcommensurabilityclasseshyperbolic}).

As in the proof of Proposition \hyperref[4.3]{4.3}, we may write $\Delta_i=\Delta_{H_i}$, and let $\gamma$ be a simple closed loop in $\Delta_1$ of length $k=|\gamma|$ based at the colored vertex with labeling $l(\gamma)\in H_1\setminus H_2$. Fix a point $p$ in the interior of the copy of $V_1$ in $M_{\Delta_1}$. We associate to $\gamma$ a closed path\[
c_{\gamma}:[0,1]\rightarrow M_{\Delta_1}\; \mathrm{with}\; c_{\gamma}(0)=c_{\gamma}(1)=p\]
such that $c_{\gamma}$ intersects the copies of the boundary submanifold $N$ transversely at times\[
0<t_1<\cdots<t_{3k}<1\]
and so that each $c_{\gamma|(t_i,t_{i+1})},0\leq i\leq 3k\,$(with $t_0=0$ and $t_{3k+1}=1$) is contained in the interior of a single building block manifold. Moreover $c_{\gamma}$ traces $\gamma$ in the obvious sense: an edge of type $a^{+1}$ in $\gamma$ corresponds to consecutive segments $[t_i,t_{i+1}],[t_{i+1},t_{i+2}]$ on which $c_{\gamma}$ travels along $A^{-}$ and then along $A^{+}$ from boundary to boundary, where both external boundaries are glued to copies of $V_1$ or $V_0$——depending on whether or not that edge is incident to colored base-point of $\gamma$.

Then we can choose a lift $\tilde{c}_{\gamma}$ of $c_{\gamma}$ to $M$ and compare the end points of two paths $\pi_2\circ\tilde{c}_{\gamma}$ and $\tilde{c}_{\gamma}$ to obtain a contradiction.
\end{proof}

\section{Explicit Constructions}
\subsection{Known Examples}

The construction in \cite{agol2018hyperbolicfourmanifoldsvanishingseibergwitten} starts from the \textit{Fibonacci manifold} $M_n$, the cyclic branched $n$-fold cover over the figure-eight knot. For $n\geq 4$, it is hyperbolic.

 From the proof of Theorem \hyperref[2.3]{2.3}, it suffices to consider the reduced invariants with rational coefficients $HM_*(Y,\fs;\Q)$. And by the universal coefficients theorem on homology, this is implied by the vanishing of $HM_*(Y,\fs;\Z/2\Z)$. Thus, it suffices to show that the computation holds with coefficients in $\Z/2\Z$.

 In \cite{agol2018hyperbolicfourmanifoldsvanishingseibergwitten}, they proved that $M_n$ is an $L$-space with coefficients in $\Z/2\Z$ for all $n$ such that $n\neq0$ mod 3, using the fact that $M_n$ is the branched double cover over the closure of the 3-braid  $(\sigma_1\sigma_2^{-1})^n$ and the surgery exact sequence \cite{kronheimer2004monopoleslensspacesurgeries}. Then they showed that $M_5$ is an arithmetic hyperbolic manifold of simplest type defined by a quadratic form over the field $\Q(\sqrt{5})$.

 Since $M_5$ is a $\Z/2\Z$ homology sphere, for $\Gamma=\pi_1(M)$, $\Gamma^{(2)}=\Gamma$, so $M_5\cong\H^n/\Gamma^{(2)}$ embeds as a totally geodesic submanifold of a close hyperbolic 4-manifold $W$ by Theorem \hyperref[3.7]{3.7}. Since the Euler number of a closed oriented hyperbolic 4-manifold is positive, we may assume that $\chi(W)>2$ by passing to a 4-fold cover and hence $b_2^+(W)>1$. Thus by Corollary \hyperref[2.4]{2.4}, these embed into a closed hyperbolic 4-manifold with vanishing Seiberg-Witten invariants.

\subsection{Construction of new manifolds}Before using the construction method in Section 4, we need to find the six building blocks first.

We recall the following commensurability criterion:

\begin{prop}\cite[Section 2.6]{PMIHES_1987__66__93_0}\label{5.1}
Let $q_1$ and $q_2$ be two quadratic forms of signature $\mathrm{(}n,1\mathrm{)}$ defined over a totally real number field $k$. Assume that every non-trivial Galois conjugate of $q_1$ as well as of $q_2$ is positive definite.

Then the two hyperbolic orbifolds with monodromy groups $SO(q_i,\mathcal{O})$ for $i=1,2$ are commensurable if and only if $q_1$ is isometric over $k$ to $\lambda q_2$ for some $\lambda\in k^*$ $\mathrm{(}i.e.$ $ Aq_1A^t=\lambda q_2$ for some $A\in GL\mathrm{(}n+1,k\mathrm{))}$.
\end{prop}

We say two such quadratic forms are similar over $k$. Thus, our attention is restricted to the quadratic forms, and we review some standard definitions and basic results about the quadratic forms in \cite{serre1973course}.

\begin{defn}
Let $k$ be a $p$-adic field, and $a,b\in k^*$. The Hilbert symbol $(a,b)_k$ is defined to be $1$ if the equation $ax^2+by^2=z^2$ has a non-trivial solution in $k$, and $-1$ otherwise. 

Given a quadratic form $q$ of rank $n+1$, define its Hasse-Witt invariant\[
\varepsilon_k(q)=\prod\limits_{i<j}(a_i,a_j)_k\in\{\pm 1\},
\]
where $a_i\in k$ and $q=a_1x_1^2+\cdots+a_{n+1}x_{n+1}^2$ in some orthogonal basis.
\end{defn}

\begin{defn}
For $u\in \Z$ the Legendre symbol $\left(\frac{u}{p}\right)$ is 0 if $u$ is divisible by $p$, 1 if the equation $u=x^2$ has a nonzero solution \rm{mod} $p$ and $-1$ otherwise.

An element $u$ with $\left(\frac{u}{p}\right)=1$ is called a quadratic residue \rm{mod} $p$.
\end{defn}

\begin{prop}\label{5.4}
Let $k$ be a $p$-adic field, and $a,b,c\in k^*$. The Hilbert symbol satisfies
\begin{enumerate}[1.]
\item $(ac,b)_k=(a,b)_k(c,b)_k$,
\item $(a^2,b)_k=1$,
\item $(a,b)_k=(a,-ab)_k$,
\end{enumerate}
and for a local field $k_{\fp}$ and $a,b\in k^{\times}_{\fp}$,
$(a,b)_{k_{\fp}}=-1$ if and only if $a$ is not a square in $k_{\fp}$ and $b$ is not a norm from $k_{\fp}(\sqrt{a})/k_{\fp}.$
\end{prop}

Hasse invariants can be used to detect the non-commensurability of quadratic forms over $p$-adic field:

\begin{prop}\cite[ Chapter IV, Section 2]{serre1973course}
For two quadratic forms $q_1$ and $q_2$ over the field $k$, a finite extension of $\Q_p$, if $q_1$ is similar $q_2$ over $k$, they have the same Hasse invariant.
\end{prop}

Let us now concentrate on arithmetic hyperbolic manifolds and recall the process of embedding arithmetic hyperbolic manifolds in Section 3. Given an admissible quadratic form f, we take $q_d=dy^2+$f, where $d$ is square-free in $k$. We assume that f has the diagonal form f=$b_1x_1^2+\cdots+b_4x_4^2$ without loss of generality.

$\Q(\sqrt{5})$ is not a $p$-adic field, so we need to calculate the Hasse invariants at some prime ideal $\fp$ of its ring of integers $\mathcal{O}=\Z\left[\frac{1+\sqrt{5}}{2}\right]$ to find the suitable $d_i\in \Q(\sqrt{5}), i=1,...,6$ such that the corresponding arithmetic hyperbolic 4-manifolds of $\mathrm{SO(}q_{d_i},\Q(\sqrt{5})\mathrm{)}$ are pairwise non-commensurable.

We now recall the definition of the \textit{Frobenius automorphism}.
For a finite field $\mathbb{F}_q$, the Frobenius automorphism is the map $x \mapsto x^q$. It generates the Galois group of $\mathbb{F}_{q^n} / \mathbb{F}_q$, which is cyclic of order $n$.

Now consider a finite Galois extension $L/K$ of number fields.
Let $\mathfrak{p}$ be an unramified prime ideal of $\mathcal{O}_K$ and choose a prime $\mathfrak{P}$ of $\mathcal{O}_L$ lying over $\mathfrak{p}$.
The Galois group $\mathrm{Gal}(L/K)$ acts transitively on such primes.
The decomposition group $D_\mathfrak{P} \subset \mathrm{Gal}(L/K)$ is the subgroup fixing $\mathfrak{P}$.
Within $D_\mathfrak{P}$, the Frobenius automorphism $\mathrm{Frob}_\mathfrak{P}$ is defined by its action on the residue field $\mathcal{O}_L / \mathfrak{P}$:
\[
\mathrm{Frob}_{\mathfrak{P}}(x)\equiv x^{\mathrm{N}_{\mathfrak{p}}}\quad \mathrm{mod}\,\mathfrak{P}\]
where $\mathrm{N}_{\mathfrak{p}}$ is the norm of $\mathfrak{p}$.

The Frobenius conjugacy class $\mathrm{Frob}_\mathfrak{p}$ in $\mathrm{Gal}(L/K)$ is the conjugacy class of this element (independent of the choice of $\mathfrak{P}$ since all such primes are conjugate).

Now suppose $L = K(\sqrt{d})$ is a quadratic Galois extension of fields, with $\mathrm{Gal}(L/K) = \{1, \sigma\}$. Let $\mathfrak{p}$ be a prime of $K$ unramified in $L$. Then the Frobenius element $\mathrm{Frob}_{\mathfrak{p}} \in \mathrm{Gal}(L/K)$ acts either trivially or nontrivially. If it acts trivially, $\mathfrak{p}$ splits in $L$ and if it acts as $\sigma$, then $\mathfrak{p}$ is inert. So the Frobenius automorphism acts trivially if and only if the defining element $d$ of $L = K(\sqrt{d})$ is a square modulo $\mathfrak{p}$.

\bigskip
The Chebotarev density theorem can guarantee the existence of suitable prime ideals. 

\begin{thm}\cite[VII.13, Chebotarev density theorem]{neukirch1999} \label{5.6}

Let $L/k$ be a finite Galois extension of number fields with Galois group $G=\Gal(L/k)$. For any conjugacy class $C\subseteq G$, the set of prime ideals $\fp$ of $k$ whose Frobenius automorphism $\mathrm{Frob_{\fp}}$ lies in $C$ has natural density $\rho$:
\[ \rho=\frac{|C|}{|G|}.
\]
In particular, such primes exist and are infinite in number.
\end{thm}

\begin{prop} \label{5.5}
For $q_d=dy^2+$\rm{f} over $\Q(\sqrt{5})$, there exist six different numbers $d_i\in \Q(\sqrt{5}), i=1,...,6$ such that the corresponding closed arithmetic hyperbolic 4-manifolds of $\mathrm{SO(}q_{d_i},\Q(\sqrt{5})\mathrm{)}$ are pairwise non-commensurable.
\end{prop}
\begin{proof}
We denote $\Q(\sqrt{5})$ by $F$ for simplicity. Square classes of $F$ are elements of $F^{\times}/(F^{\times})^2$, and there are infinitely many square classes for $F=\Q(\sqrt{5})$. So we can choose $d_i, i=1,...,6$ so that $d_i$ and $d_j$ are not in the same square class of $F^{\times}$ for $i\neq j$ and that $b_1$ and $d_i^{-1}d_j$ are not in the same square class of $F$ for $i\neq j$.

Fix $i\neq j$. Let $L_1=F(\sqrt{d_i^{-1}d_j})$, $L_2=F(\sqrt{b_1},...,\sqrt{b_4})$ and $L=L_1\cdot L_2$ be the compositum of the two finite extensions. $L/F$ is a Galois extension, and Theorem \hyperref[5.6]{5.6} ensures the existence of some prime ideal $\fp$ where:
\begin{enumerate}[1.]
\item $\mathrm{Frob}_{\fp}$ acts as nontrivial on $F(\sqrt{d_i^{-1}d_j})$ (so $d_i^{-1}d_j$ is a non-square in $F_{\fp}$).
\item $\mathrm{Frob}_{\fp}$ acts as nontrivial on $F(\sqrt{b_1})$ (so $b_1$ is a non-square in $F_{\fp}$).
\item $\mathrm{Frob}_{\fp}$ acts as nontrivial on $F(\sqrt{b_1d_id_j^{-1}})$ (so $b_1$ is a non-square in $F_{\fp}$ and $b_1$ is not a norm from $F(\sqrt{d_i^{-1}d_j})/F$).
\item $\mathrm{Frob}_{\fp}$ acts as trivial on $F(\sqrt{b_l})$ for $l\neq 1$ (so $b_l$ are squares in $F_{\fp}$).
\end{enumerate}

By Proposition \hyperref[5.4]{5.4}, we obtain $(b_1, d_i^{-1}d_j)_{F_{\fp}}=-1$ and $(b_l, d_i^{-1}d_j)_{F_{\fp}}=1$ for $l\neq 1$, so we know that $\varepsilon_{F_{\fp}}(q_{d_i})\neq \varepsilon_{F_{\fp}}(q_{d_j})$. By Proposition \hyperref[5.5]{5.5}, the two quadratic forms $q_{d_i}$ and $q_{d_j}$ are not similar over $F_{\fp}$, and hence a fortiori $q_{d_i}$ and $q_{d_j}$ are not similar over $F$. Since $i,j$ are arbitrary, the quadratic forms $q_{d_i}$ are pairwise not similar over $F$, and by Proposition \hyperref[5.1]{5.1}, we prove the proposition.
\end{proof}

By taking the double covers of the first four closed arithmetic hyperbolic 4-manifolds as in Corollary \hyperref[2.4]{2.4} if necessary, we obtain four non-commensurable closed arithmetic hyperbolic 4-manifolds $W_{d_i}, i=1,...,4$ containing the $L$-space $M_5$ as a non-separating hypersurface. Cut $W_{d_i},i=1,...,4$ along one copy of $M_5$, and then we obtain the building blocks $A^{\pm}$ and $B^{\pm}$.
They are the ones placed on the edges of the finite decorated graph.

Recall the process of embedding arithmetic hyperbolic manifolds in Section 3. Given an arithmetic hyperbolic $(n-1)$-manifold $N$, we embed $N$ into an arithmetic hyperbolic manifold $M=\H^n/\Gamma$ by considering the $q$-hyperboloid model for $\H^n$. Let $R=\H^n\cap\{x\in\R^{n+1}:x_0=0\}$ and $N=R/\Gamma_0$, we find $\Gamma$ such that $\Gamma_0\subseteq\Gamma$ is the subgroup of $\Gamma$ consisting of elements that preserve $R$ and embed $N$ into $M=\H^n/\Gamma$.

To obtain building blocks with four totally geodesic boundary components, we need the following proposition:

\begin{prop}\cite[Proposition 4.3]{gelander2014countingcommensurabilityclasseshyperbolic}\label{5.8}
 For every $m\in\Z$ there exists a finite normal cover $M'$ of $M$ that contains \rm{(}at least\rm{)} $m$ disjoint copies $N_1,...,N_m$ of $N$ such that $M'\setminus \mathop{\cup}\limits_{i=1}^m N_i$ is connected.
\end{prop}

\begin{proof}
Suppose first that $N$ is separating in $M$. It follows that $\Gamma$ is isomorphic to the amalgamated product $\Gamma_1*_{\Gamma_0}\Gamma_2$, where $\Gamma_1$ and $\Gamma_2$ are the fundamental groups of the two connected components of $M\setminus N$. By \cite[Section 0.1]{PMIHES_1987__66__93_0} the subgroups $\Gamma_i, i=1,2$ are Zariski dense in $\mathrm{SO}(n,1)$ and their Zariski closures in $\mathrm{SO}(n+1,\C)$ are also semisimple algebraic groups. Since $\Gamma_i\subset\mathrm{SO}(n,1)\subset\mathrm{SO}(n+1,\C)$ and that $\mathrm{SO(}n+1,\C\mathrm{)}$ is an order 2 quotient of its universal covering, it follows from the Weisfeiler–Nori strong approximation theorem (see \cite{Rapinchuk2013}) that each $\Gamma_i, i=1,2$ is mapped to a subgroup of index at most 2 in almost every congruence quotient of $\Gamma$ by considering the lift of $\Gamma_i$ to Spin($n+1,\C$).

Since $\Gamma_0$ is the intersection of $\Gamma$ with a parabolic subgroup\[
\Gamma_0 =\{\gamma\in\Gamma:\gamma=
\begin{pmatrix}
1 & 0 & \cdots &  0\\
* & * &  *   & * \\
\vdots & * & \ddots & *\\
* & * & * & *\\
\end{pmatrix}\},
\]
it is clear that we may find congruence quotients of $\Gamma$ in which the image of $\Gamma_0$ is of arbitrarily large index. Let $\Gamma(p)$ be principal congruence subgroup in $\Gamma$ such that\[
[\Gamma_i\cap\Gamma(p):\Gamma_0\cap\Gamma(p)]=k_i\geq 3, \,\mathrm{for}\,i=1,2
\]
and denote by $\bar{\Gamma}_i, i=0,1,2$ the image of $\Gamma_i$ in the finite group $\Gamma/\Gamma(p)$, respectively. Set $\Lambda=\bar{\Gamma}_1*_{\bar{\Gamma}_0}\bar{\Gamma}_2$ and consider the map\[
\pi :\Gamma=\Gamma_1*_{\Gamma_0}\Gamma_2\rightarrow \bar{\Gamma}_1*_{\bar{\Gamma}_0}\bar{\Gamma}_2=\Lambda.\]

According to the Bass-Serre theory (see \cite[Chapter I, Theorem 5.1]{Bass1993Covering}), the group $\Lambda$ acts on the $(k_1,k_2)$-bi-regular tree $T$. It is well known (see \cite[p. 120]{Serre1980}) that $\Lambda$ has a finite index free subgroup $\Lambda'$ acting freely on $T$ with $\Lambda'\setminus T$ being a $(k_1,k_2)$-bi-regular finite graph. By taking a further finite index subgroup $\Lambda''$ we can assume that $\Lambda''$ is normal in $\Lambda$ and of rank at least $m$. It follows that the graph $\Lambda''\setminus T$ has at least $m$ simultaneously non-separating edges.

The group $\Gamma$ acts on $T$ as well via the map $\pi$. Let $\Gamma''=\pi^{-1}(\Lambda '')\unlhd\Gamma$. As $\Gamma''$ acts on $T$ with the same fundamental domain as $\Lambda''$, it splits as a graph of groups over the graph $\Lambda''\setminus T$. Moreover, this graph of groups covers the graph of groups of $\Lambda=\bar{\Gamma}_1*_{\bar{\Gamma}_0}\bar{\Gamma}_2$
(see \cite[Section 4]{Bass1993Covering}).

To complete the proof in this case, let $M'$ be the normal cover of $M$ corresponding to $\Gamma''$. The connected components of the preimage of $N$ inside $M'$ serve as edges in a decomposition of $M'$ according to the graph structure of $\Gamma''\setminus T$. Since $M'$ is normal, it is clear from the construction that all these connected components are isometric to $N$. Moreover as $N$ embeds in $M$, every two of them are disjoint. The result follows by taking copies of $N$ which correspond to a jointly non-separating set of $m$ edges of $\Gamma''\setminus T$.

The remaining case where $M$ is non-separating is dealt with by a similar argument. In that case $\Gamma$ is isomorphic to the HNN extension $\Gamma_1*_{\Gamma_0}$, where $\Gamma_1$ is the fundamental group of $M\setminus N$. The map $\pi$ is defined in an analogous fashion, and essentially the same proof goes through. 
\end{proof}

By the above proposition, we can obtain the building blocks $V_0$ and $V_1$ by taking the case $2m=4$ for $N=M_5$ and $M=W_{d_i}, i=5,6$. So we have constructed the six building blocks required in Section 4.

The following proposition tells the number of decorated graphs up to isomorphism:

\begin{prop}\cite[Ch. 2]{lubotzky_segal2003}\label{5.9}
Let $a_n$ denote the number of subgroups of index $n$ in the free group $F$ on two generators. Then $a_n \geq n^{\frac{n}{2}}$ for every $n$.
\end{prop}

Let $V_0$ denote the maximal volume of the six building blocks. For each representative $F'$ of the isomorphism classes of subgroups of $F$ of the index $n$, the corresponding $M_\Delta$ has volume $V\leq 5nV_0$. 

So for given volume $\frac{V}{2}$, we can construct at least $[V/10V_0]^{\frac{[V/10V_0]}{2}}$ closed hyperbolic 4-mainfolds with volume below $\frac{V}{2}$ containing the non-separating $L$-space $M_5$ by Proposition \hyperref[5.9]{5.9}. They are pairwise non-commensurable by Proposition \hyperref[4.4]{4.4}, and by Corollary \hyperref[2.4]{2.4} we can find $[V/10V_0]^{\frac{[V/10V_0]}{2}}$ pairwise non-commensurable closed hyperbolic 4-mainfolds at volume most $V$ with vanishing Seiberg-Witten invariants. By substituting \(n = \lfloor V/10V_0 \rfloor\) into Proposition 5.9, we derive the lower bound \(VC(v) \geq v^{cv}\) for \(v\) sufficiently large.

The upper bound $v^{av}$ is exactly from \cite[Theorem 1.1]{gelander2014countingcommensurabilityclasseshyperbolic}, so we prove Theorem \hyperref[1.1]{1.2}.

\section*{Acknowledgments} 
 I am extremely grateful to my mentor, Ben Lowe, who helped me discover this fascinating topic and provided feedback on the paper. I extend my deepest appreciation to Peter May for organizing this exceptional REU program and giving me valuable advice on the paper.

\newpage

\bibliographystyle{alpha}
\bibliography{reference}

\end{document}